\documentclass{article}
\usepackage{amsmath}
\usepackage{amsfonts}

\newtheorem{theorem}{Theorem}[section]
\newtheorem{definition}[theorem]{Definition}
\newtheorem{lemma}[theorem]{Lemma}

\newtheorem{proposition}[theorem]{Proposition}

\newenvironment{remark}[1]{\medskip\par\noindent{\bf #1.}\rm}
{\medskip\par\noindent}

\newenvironment{proof}[1]{\medskip\par\noindent{\sc #1.\ }}
{~\rule{0.5em}{0.5em}\medskip\par}

\newcommand{\sgn}{{\rm sign}\hskip0.02cm}
\def\vol{\hskip0.02cm{\rm vol}\hskip0.01cm}

\newcommand{\lin}{{\rm lin}\hskip0.02cm}

\begin{document}
\title{\LARGE{\bf{Auerbach bases and minimal volume sufficient enlargements}}}

\author{M.~I.~Ostrovskii}

\date{\today}
\maketitle

\begin{large}

\noindent{\bf Abstract.} Let $B_Y$ denote the unit ball of a
normed linear space $Y$. A symmetric, bounded, closed, convex set
$A$ in a finite dimensional normed linear space $X$ is called a
{\it sufficient enlargement} for $X$ if, for an arbitrary
isometric embedding of $X$ into a Banach space $Y$, there exists a
linear projection $P:Y\to X$ such that $P(B_Y)\subset A$. Each
finite dimensional normed space has a minimal-volume sufficient
enlargement which is a parallelepiped, some spaces have ``exotic''
minimal-volume sufficient enlargements. The main result of the
paper is a characterization of spaces having ``exotic''
minimal-volume sufficient enlargements in terms of Auerbach
bases.\bigskip

\noindent{\bf 2000 Mathematics Subject Classification:} 46B07
(primary), 52A21, 46B15 (secondary).

\section{Introduction}

All linear spaces considered in this paper will be over the reals.
By a {\it space} we mean a normed linear space, unless it is
explicitly mentioned otherwise. We denote by $B_X$ the closed unit
ball of a space $X$. We say that subsets $A$ and $B$ of finite
dimensional linear spaces $X$ and $Y$, respectively, are {\it
linearly equivalent} if there exists a linear isomorphism $T$
between the subspace spanned by $A$ in $X$ and the subspace
spanned by $B$ in $Y$ such that $T(A)=B$. By a {\it symmetric} set
$K$ in a linear space we mean a set such that $x\in K$ implies
$-x\in K$.
\medskip

Our terminology and notation of Banach space theory follows
\cite{JL01}. By $B_p^n$, $1\le p\le\infty$, $n\in \mathbb{N}$ we
denote the closed unit ball of $\ell_p^n$. Our terminology and
notation of convex geometry follows \cite{Sch93}. A Minkowski sum
of finitely many line segments is called a {\it zonotope}.
\medskip

We use the term {\it ball}~ for a symmetric, bounded, closed,
convex set with interior points in a finite dimensional linear
space.

\begin{definition} {\rm\cite{Ost96} A ball in a finite dimensional normed space $X$ is called a
{\it sufficient enlargement} (SE) for $X$ (or of $B_X$) if, for an
arbitrary isometric embedding of $X$ into a Banach space $Y$,
there exists a projection $P:Y\to X$ such that $P(B_Y)\subset A$.
A sufficient enlargement $A$ for $X$ is called a {\it
minimal-volume sufficient enlargement} (MVSE) if $\vol A\le\vol D$
for each SE $D$ for $X$.}
\end{definition}

It was proved in \cite[Theorem 3]{Ost08} that each MVSE is a
zonotope generated by a totally unimodular matrix and the set of
all MVSE (for all spaces) coincides with the set of all space
tiling zonotopes which was described in \cite{Erd99},
\cite{McM75}. It is known (see \cite[Theorem 6]{Ost98}, the result
is implicit in \cite[pp.~95--97]{GMP96}) that a minimum-volume
parallelepiped containing $B_X$ is an MVSE for $X$. It was
discovered (see \cite[Theorem 4]{Ost04} and \cite[Theorem
4]{Ost08}) that spaces $X$ having a non-parallelepipedal MVSE are
rather special: they should have a two-dimensional subspace whose
unit ball is linearly equivalent to a regular hexagon. In
dimension two this provides a complete characterization (see
\cite{Ost04}). On the other hand, the unit ball of
$\ell_\infty^n$, $n\ge 3$, has a regular hexagonal section, but
the only MVSE for $\ell_\infty^n$ is its unit ball (so it is a
parallelepiped). A natural problem arises: To characterize Banach
spaces having non-parallelepipedal MVSE in dimensions $d\ge 3$.
The main purpose of this paper is to characterize such spaces in
terms of Auerbach bases. At the end of the paper we make some
remarks on MVSE for $\ell_1^n$ and study relations between the
class of spaces having non-parallelepipedal MVSE and the class of
spaces having a $1$-complemented subspace whose unit ball is
linearly equivalent to a regular hexagon.

\section{Auerbach bases}

We need to recall some well-known results on bases in finite
dimensional normed spaces. Let $X$ be an $n$-dimensional normed
linear space. For a vector $x\in X$ by $[-x,x]$ we denote the line
segment joining $-x$ and $x$. For $x_1,\dots,x_k\in X$ by
$M(\{x_i\}_{i=1}^k)$ we denote the Minkowski sum of the
corresponding line segments, that is,
$$M(\{x_i\}_{i=1}^k)=\{x:~x=y_1+\dots+y_k \hbox{ for some
}y_i\in[-x_i,x_i],~i=1,\dots,k\}.$$

Let $\{x_i\}_{i=1}^n$ be a basis in $X$, its {\it biorthogonal
functionals} are defined by $x_i^*(x_j)=\delta_{ij}$ (Kronecker
delta). The basis $\{x_i\}_{i=1}^n$ is called an {\it Auerbach
basis} if $||x_i||=||x^*_i||=1$ for all $i\in\{1,\dots,n\}$.
According to \cite[Remarks to Chapter VII]{Ban32} H.~Auerbach
proved the existence of such bases for each finite dimensional
$X$.

\begin{remark}{Historical comment} The book \cite{Ban32} does not
contain any proofs of the existence of Auerbach bases. The two
dimensional case of Auerbach's result was proved in \cite{Aue30}.
Unfortunately Auerbach's original proof in the general case seems
to be lost. Proofs of the existence of Auerbach bases discussed
below are taken from \cite{Day47} and \cite{Tay47}. The paper
\cite{Pli95} contains interesting results on relation between
upper and lower Auerbach bases (which are defined below) and
related references.
\end{remark}

It is useful for us to recall the standard argument for proving
the existence of Auerbach bases (it goes back at least to
\cite{Tay47}). Consider the set $N(=N(X))$ consisting of all
subsets $\{x_i\}_{i=1}^n\subset X$ satisfying $||x_i||=1$,
$i\in\{1,\dots,n\}$. It is a compact set in its natural topology;
and the $n$-dimensional volume of $M(\{x_i\}_{i=1}^n)$ is a
continuous function on $N$. Hence it attains its maximum on $N$.
Let $U\subset N$ be the set of $n$-tuples on which the maximum is
attained. It is easy to see that each $\{x_i\}_{i=1}^n\in U$ is a
basis (for linearly dependent sets the volume is zero). Another
important observation is that $M(\{x_i\}_{i=1}^n)\supset B_X$ if
$\{x_i\}_{i=1}^n\in U$. In fact, if there is $y\in B_X\backslash
M(\{x_i\}_{i=1}^n)$ then (since the volume of a parallelepiped is
the product of the length of its height and the
$(n-1)$-dimensional volume of its base), there is
$i\in\{1,\dots,n\}$ such that replacing $x_i$ by $y$ we get a
parallelepiped whose volume is strictly greater the volume of
$M(\{x_i\}_{i=1}^n)$. Since we may assume $||y||=1$, this is a
contradiction with the definition of $U$.
\medskip

The following lemma shows that each basis from $U$ is an Auerbach
basis.

\begin{lemma}\label{L:Auer} A system $\{x_i\}_{i=1}^n\in N$ is an Auerbach basis if and only
if $M(\{x_i\}_{i=1}^n)\supset B_X$.
\end{lemma}

\begin{proof}{Proof} It is easy to see that
$$M(\{x_i\}_{i=1}^n)=\{x:~ |x_i^*(x)|\le 1\hbox{ for
}i=1,\dots,n\}$$ for each basis $\{x_i\}_{i=1}^n$. Hence
$M(\{x_i\}_{i=1}^n)\supset B_X$ if and only if $||x_i^*||\le 1$
for each $i$. It remains to observe that the equality $||x_i||=1$
implies $||x_i^*||\ge 1$, $i=1,\dots,n$.\end{proof}

This result justifies the following definition.

\begin{definition} {\rm A basis from $U$ is called an {\it
upper Auerbach basis}.}
\end{definition}

Another way of showing that each finite dimensional space $X$ has
an Auerbach basis was discovered in \cite{Day47} (see also
\cite{PS91}). It was proved that each parallelepiped $P$
containing $B_X$ and having the minimum possible volume among all
parallelepipeds containing $B_X$ is of the form
$M(\{x_i\}_{i=1}^n)$ for some $\{x_i\}_{i=1}^n\in N(X)$. By Lemma
\ref{L:Auer} the corresponding system $\{x_i\}_{i=1}^n$ is an
Auerbach basis.

\begin{definition} {\rm A basis $\{x_i\}_{i=1}^n$ for which
$M(\{x_i\}_{i=1}^n)$ is one of the minimum-volume parallelepipeds
containing $B_X$ is called a {\it lower Auerbach basis}.}
\end{definition}

The notions of lower and upper Auerbach bases are dual to each
other.

\begin{proposition}\label{P:HL} A basis $\{x_i\}_{i=1}^n$ in $X$ is a lower
Auerbach basis if and only if the biorthogonal sequence
$\{x_i^*\}_{i=1}^n$ is an upper Auerbach basis in $X^*$.
\end{proposition}

\begin{proof}{Proof} We choose a basis $\{e_i\}_{i=1}^n$ in $X$
and let $\{e_i^*\}_{i=1}^n$ be its biorthogonal functionals in
$X^*$. We normalize all volumes in $X$ in such a way that the
volume of $M(\{e_i\}_{i=1}^n)$ is equal to $1$ and all volumes in
$X^*$ in such a way that the volume of $M(\{e^*_i\}_{i=1}^n)$ is
equal to $1$ (one can see that normalizations do not matter for
our purposes).

Let $K=(x_{i,j})_{i,j=1}^n$ be the matrix whose columns are
coordinates of an Auerbach basis $\{x_j\}_{j=1}^n$ with respect to
$\{e_i\}_{i=1}^n$; and let $K^*=(x^*_{i,j})_{i,j=1}^n$ be a matrix
whose rows are coordinates of $\{x^*_i\}_{i=1}^n$ (which is an
Auerbach basis in $X^*$) with respect to $\{e^*_j\}_{j=1}^n$. Then
$K^*\cdot K=I$ (the identity matrix). Therefore
$$|\det K^*|\cdot|\det K|=1.$$
Hence $\vol(M(\{x_i\}_{i=1}^n)\cdot \vol(M(\{x^*_i\}_{i=1}^n)=1$,
and one of these volumes attains its maximum on the set of
Auerbach bases if and only if the other attains its minimum.
\end{proof}

\section{The main result}

\begin{theorem}\label{T:main} An $n$-dimensional normed linear space $X$ has a
non-parallelepipedal MVSE if and only if $X$ has a lower Auerbach
basis $\{x_i\}_{i=1}^n$ such that the unit ball of the
two-dimensional subspace $\lin\{x_1,x_2\}$ is linearly equivalent
to a regular hexagon.
\end{theorem}

\begin{proof}{Proof. ``Only if'' part} We start by considering the
case when the space $X$ is polyhedral, that is, when $B_X$ is a
polytope. In this case we may consider $X$ as a subspace of
$\ell_\infty^m$ for some $m\in {\mathbb{N}}$. Since $X$ has an
MVSE which is not a parallelepiped, there exists a linear
projection $P:\ell_\infty^m \to X$ such that $P(B_\infty^m)$ has
the minimal possible volume, but $P(B_\infty^m)$ is not a
parallelepiped. We consider the standard Euclidean structure on
$\ell_\infty^m$. Let $\{q_1,\dots,q_{m-n}\}$ be an orthonormal
basis in $\ker P$ and let $\{\tilde q_1,\dots,\tilde q_n\}$ be an
orthonormal basis in the orthogonal complement of $\ker P$. As it
was shown in \cite[Lemma 2]{Ost03}, $P(B_\infty^m)$ is linearly
equivalent to the zonotope spanned by rows of $\tilde Q=[\tilde
q_1,\dots,\tilde q_n]$. By the assumption this zonotope is not a
parallelepiped. It is easy to see that this assumption is
equivalent to: there exists a minimal linearly dependent
collection of rows of $\tilde Q$ containing $\ge 3$ rows. This
condition implies that we can reorder the coordinates in
$\ell_\infty^m$ and multiply the matrix $\tilde Q$ from the right
by an invertible $n\times n$ matrix $C_1$ in such a way that
$\tilde QC_1$ has a submatrix of the form
$$\left(
\begin{array}{cccc}
1 & 0 & \dots & 0\\
0 & 1 & \dots & 0\\
\vdots & \vdots & \ddots & \vdots\\
0 & 0 & \dots & 1\\
a_1 & a_2 & \dots & a_n
\end{array}\right),
$$
where $a_1\ne 0$ and $a_2\ne 0$. Let $\mathcal{X}$ be an $m\times
n$ matrix whose columns form a basis of $X$ (considered as a
subspace of $\ell_\infty^m$). The argument of \cite{Ost03} (see
the conditions (1)--(3) on p.~96) implies that $\mathcal{X}$ can
be multiplied from the right by an invertible $n\times n$ matrix
$C_2$ in such a way that $\mathcal{X}C_2$ is of the form
$$\left(
\begin{array}{cccc}
1 & 0 & \dots & 0\\
0 & 1 & \dots & 0\\
\vdots & \vdots & \ddots & \vdots\\
0 & 0 & \dots & 1\\
\sgn a_1 & \sgn a_2 & \dots & *\\
\vdots & \vdots & \ddots & \vdots
\end{array}\right),
$$
where at the top there is an $n\times n$ identity matrix, and all
minors of the matrix $\mathcal{X}C_2$ have absolute values $\le
1$.
\medskip

Observe that columns on $\mathcal{X}C_2$ also form a basis in $X$.
Changing signs of the first two columns and of the first two
coordinates of $\ell_\infty^m$, if necessary, we get that the
subspace $X\subset\ell_\infty^m$ is spanned by columns of the
matrix
\begin{equation}\label{Z}
\left(
\begin{array}{lllll}
1 & 0 & 0 &\dots & 0\\
0 & 1 & 0 &\dots & 0\\
0 & 0 & 1 &\dots & 0\\
\vdots & \vdots & \vdots & \ddots & \vdots\\
0 & 0 & 0 &\dots & 1\\
1 & 1 & b_{n+1,3} & \dots & b_{n+1,n}\\
b_{n+2,1} & b_{n+2,2} & * & \dots & *\\
\vdots & \vdots & \vdots & \ddots & \vdots\\
b_{m,1} & b_{m,2} & * & \dots & *
\end{array}\right),
\end{equation}
in which absolute values of all minors are $\le 1$. This
restriction on minors implies $|b_{i,1}-b_{i,2}|\le 1$,
$|b_{i,1}|\le 1$, and $|b_{i,2}|\le 1$. A routine verification
shows that these inequalities imply that the first two columns
span a subspace of $X\subset\ell_\infty^m$ whose unit ball is
linearly equivalent to a regular hexagon (see \cite[p.~390]{Ost04}
for more details).
\medskip

It remains to show that the columns of \eqref{Z} form a lower
Auerbach basis in $X$. Let us denote the columns of \eqref{Z} by
$\{x_i\}_{i=1}^n$ and the biorthogonal functionals of
$\{x_i\}_{i=1}^n$ (considered as vectors in $X^*$) by
$\{x^*_i\}_{i=1}^n$.
\medskip

We map $\{x^*_i\}_{i=1}^n$ onto the unit vector basis of
$\mathbb{R}^n$. This mapping maps $B_{X^*}$ onto the symmetric
convex hull of vectors whose coordinates are rows of the matrix
\eqref{Z}. In fact, using the definitions we get
$$\left\|\sum_{i=1}^n\alpha_ix^*_i\right\|_{X^*}=
\max\left\{\left|\sum_{i=1}^n \alpha_i\beta_i\right|:~ \max_{1\le
j\le m}\left|\sum_{i=1}^n\beta_ib_{ji}\right|\le 1\right\}.
$$
Therefore, if $\{\alpha_i\}_{i=1}^n\in\mathbb{R}^n$ is in the
symmetric convex hull of $\{b_{ji}\}_{i=1}^n\in\mathbb{R}^n$,
$j=1,\dots,m$, then $$\left|\sum_{i=1}^n \alpha_i\beta_i\right|\le
\max_{1\le j\le m}\left|\sum_{i=1}^n\beta_ib_{ji}\right|\hbox{ and
 }\left\|\sum_{i=1}^n\alpha_ix^*_i\right\|_{X^*}\le 1.$$

On the other hand, if $\{\alpha_i\}_{i=1}^n$ is not in the
symmetric convex hull of $\{b_{ji}\}_{i=1}^n\in\mathbb{R}^n$,
$j=1,\dots,m$, then, by the separation theorem (see, e.g.
\cite[Theorem 1.3.4]{Sch93}), there is $\{\beta_i\}_{i=1}^n$ such
that
$$\max_{1\le j\le
m}\left|\sum_{i=1}^n\beta_ib_{ji}\right|\le 1,\hbox{ but }
\left|\sum_{i=1}^n \alpha_i\beta_i\right|>1,$$ and hence $$
\left\|\sum_{i=1}^n\alpha_ix^*_i\right\|_{X^*}> 1.$$

Thus the restriction on the absolute values of minors of \eqref{Z}
implies that $\{x^*_i\}_{i=1}^n$ is an upper Auerbach basis in
$X^*$. By Proposition \ref{P:HL}, $\{x_i\}_{i=1}^n$ is a lower
Auerbach basis in $X$.
\medskip

Now we consider the general case. Let $Y$ be an $n$-dimensional
space and $A$ be a non-parallelepipedal MVSE for $Y$. By
\cite[Theorem 3]{Ost08} and \cite[Lemma 1]{Ost04} there is a
polyhedral space $X$ such that $B_X\supset B_Y$ and $A$ is an SE
(hence MVSE) for $X$. By the first part of the proof there is a
lower Auerbach basis $\{x_i\}_{i=1}^n$ in $X$ such that the unit
ball of the subspace of $X$ spanned by $\{x_1,x_2\}$ is linearly
equivalent to a regular hexagon. The basis $\{x_i\}_{i=1}^n$ is a
lower Auerbach basis for $Y$ too. In fact, the spaces have the
same MVSE, hence a minimum-volume parallelepiped containing $B_X$
is a also a minimum-volume parallelepiped containing $B_Y$. It
remains to show that the unit ball of the subspace spanned in $Y$
by $\{x_1,x_2\}$ is also a regular hexagon.
\medskip

To achieve this goal we use an additional information about the
basis $\{x_i\}$ which we get from the first part of the proof.
Namely, we use the observation that the vertices of the unit ball
of the subspace $\hbox{lin}(x_1,x_2)$ are: $\pm x_1$, $\pm x_2$,
$\pm (x_1-x_2)$. So it remains to show that $(x_1-x_2)\in B_Y$.
This has already been done in \cite[pp.~617--618]{Ost08}.
\medskip

\noindent{\sc ``If'' part.} First we consider the case when $X$ is
polyhedral. Suppose that $X$ has a lower Auerbach basis
$\{x_i\}_{i=1}^n$ and that $x_1,x_2$ span a subspace whose unit
ball is linearly equivalent to a regular hexagon. Then the
biorthogonal functionals $\{x_i^*\}_{i=1}^n$ form an upper
Auerbach basis in $X^*$. We join to this sequence all extreme
points of $B_{X^*}$. Since $X$ is polyhedral, we get a finite
sequence which we denote $\{x_i^*\}_{i=1}^m$. Then
$$x\mapsto \{x^*_i(x)\}_{i=1}^m$$
is an isometric embedding of $X$ into $\ell_\infty^m$. Writing
images of $\{x_i\}_{i=1}^n$ as columns, we get a matrix of the
form:
\begin{equation}\label{X}
(b_{ij})=\left(
\begin{array}{ccccc}
1 & 0 & 0 &\dots & 0\\
0 & 1 & 0 &\dots & 0\\
0 & 0 & 1 &\dots & 0\\
\vdots & \vdots & \vdots & \ddots & \vdots\\
0 & 0 & 0 &\dots & 1\\
b_{n+1,1} & b_{n+1,2} & * & \dots & *\\
b_{n+2,1} & b_{n+2,2} & * & \dots & *\\
\vdots & \vdots & \vdots & \ddots & \vdots\\
b_{m,1} & b_{m,2} & * & \dots & *
\end{array}\right).
\end{equation}

Since $\{x_i^*\}_{i=1}^n$ is an upper Auerbach basis, absolute
values of all minors of this matrix do not exceed $1$.
\medskip

Now we use fact that the linear span of $\{x_1,x_2\}$ is a regular
hexagonal space in order to show that we may assume that at least
one of the pairs $(b_{k,1},b_{k,2})$ in \eqref{X} is of the form
$(\pm1,\pm1)$. (Sometimes we need to modify the matrix \eqref{X}
to achieve this goal.)
\medskip

The definition of the norm on $\ell_\infty^m$ implies that there
is a $3\times 2$ submatrix $S$ of the matrix $(b_{i,j})$
$(i=1,\dots,m,~ j=1,2)$ whose columns span a regular hexagonal
subspace in $\ell_\infty^3$, and for each
$\alpha_1,\alpha_2\in\mathbb{R}$ the equality
\begin{equation}\label{E:max}\max_{1\le i\le m}|\alpha_1b_{i,1}+\alpha_2b_{i,2}|=
\max_{i\in A}|\alpha_1b_{i,1}+\alpha_2b_{i,2}|\end{equation}
holds, where $A$ is the set of labels of rows of $S$.
\medskip

To find such a set $S$ we observe that for each side of the
hexagon we can find $i\in\{1,\dots,m\}$ such that the side is
contained in the set of vectors of $\ell_\infty^m$ for which the
$i^{th}$ coordinate is either $1$ or $-1$ (this happens because
the hexagon is the intersection of the unit sphere of
$\ell_\infty^n$ with the two dimensional subspace). Picking one
side from each symmetric with respect to the origin pair of sides
and choosing (in the way described above) one label for each of
the pairs, we get the desired set $A$. To see that it satisfies
the stated conditions we consider the operator
$R:\ell_\infty^m\to\ell_\infty^3$ given by
$R(\{x_i\}_{i=1}^n)=\{x_i\}_{i\in A}$. The stated condition can be
described as: the restriction of $R$ to the linear span of the
first two columns of the matrix $(b_{ij})$ is an isometry. To show
this it suffices to show that a vector of norm $1$ is mapped to a
vector of norm $1$. This happens due to the construction of $A$.
\medskip

It is clear from \eqref{X} that the maximum in the left hand side
of \eqref{E:max} is at least $$\max\{|\alpha_1|,|\alpha_2|\}.$$
Hence at least one of the elements in each of the columns of $S$
is equal to $\pm 1$. A (described below) simple variational
argument shows that changing signs of rows of $S$, if necessary,
we may assume that
\medskip

\noindent{\bf (1)} Either $S$ contains a row of the form $(1,0)$
or two rows of the forms $(1,a)$ and $(1,-b)$, $a,b>0$.
\medskip

\noindent{\bf (2)} Either $S$ contains a row of the form $(0,1)$
or two rows of the forms $(c,1)$ and $(-d,1)$, $c,d>0$.\medskip

\noindent{\bf Note.} At this point we allow the changes of signs
needed for {\bf (1)} and for {\bf (2)} to be different.
\medskip

The mentioned above variational argument consists of showing that
in the cases when {\bf (1)} and {\bf (2)} are not satisfied there
are $\alpha_1,\alpha_2\in\mathbb{R}$ such that $$\max_{i\in
A}|\alpha_1b_{i,1}+\alpha_2b_{i,2}|<\max\{|\alpha_1|,|\alpha_2|\}.$$
Let us describe the argument in one of the typical cases (all
other cases can be treated similarly).
\medskip

Suppose that $S$ is such that all entries in the first column are
positive, $S$ contains a row of the form $(1,b)$ with $b>0$, but
not a row of the form $(1,a)$ with $a\le 0$ (recall that absolute
values of entries of \eqref{X} do not exceed $1$). It is clear
that we get the desired pair by letting $\alpha_1=1$ and choosing
$\alpha_2<0$ sufficiently close to $0$.
\medskip

The restriction on the absolute values of the determinants implies
that if the second alternative holds in {\bf (1)}, then $a+b\le 1$
and if the second alternative holds in {\bf (2)}, then $c+d\le 1$.
This implies that the second alternative cannot hold
simultaneously for {\bf (1)} and {\bf (2)}, and thus, there is a
no need in different changes of signs for {\bf (1)} and {\bf
(2)}.\medskip

Therefore it suffices to consider two cases:\medskip

{\bf I.} The matrix $S$ is of the form

\begin{equation}\label{E:I}
\left(
\begin{array}{cc}
1 & 0 \\
0 & 1 \\
u & v
\end{array}\right).
\end{equation}

{\bf II.} The matrix $S$ is of the form

\begin{equation}\label{E:II}
\left(
\begin{array}{rc}
1 & 0 \\
c & 1 \\
-d & 1
\end{array}\right).
\end{equation}

Let us show that the fact that the columns of $S$ span a regular
hexagonal space implies that all of its $2\times 2$ minors have
the same absolute values. It suffices to do this for any basis of
the same subspace of $\ell_\infty^3$. The subspace should
intersect two adjacent edges of the cube. Changing signs of the
unit vector basis in $\ell_\infty^n$, if necessary, we may assume
that the points of intersection are of the forms
\begin{equation}\label{E:vec1}\left(\begin{array}{r}
1 \\
1 \\
\alpha
\end{array}\right) \hbox{ and }
\left(\begin{array}{r}
\beta \\
1 \\
1
\end{array}\right),
~ |\alpha|< 1,~ |\beta|< 1.
\end{equation}
The points of intersection are vertices of the hexagon. One more
vertex of the hexagon is a vector of the form
\begin{equation}\label{E:vec2}\left(
\begin{array}{r}
-1  \\
\gamma \\
1
\end{array}\right),~ |\gamma|<1.
\end{equation}
If the hexagon is linearly equivalent to the regular, then all
parallelograms determined by pairs of vectors of the triple
described in \eqref{E:vec1} and \eqref{E:vec2} should have equal
areas. Therefore the determinants of matrices formed by a unit
vector and two of the vectors from the triple described in
\eqref{E:vec1} and \eqref{E:vec2} should have the same absolute
values. It is easy to see that the obtained equalities imply
$\alpha=\beta=0$. The conclusion follows.
\medskip

In the case {\bf I} the equality of $2\times 2$ minors implies
that $|u|=|v|=1$, and we have found a $(\pm1,\pm1)$ row.
\medskip

In the case {\bf II} we derive $c+d=1$. Now we replace the element
$x_1$ in the basis consisting of columns of \eqref{X} by
$x_1-cx_2$. It is clear that the sequence we get is still a basis
in the same space, and this modification does not change values of
minors of sizes at least $2\times 2$. As for minors of sizes
$1\times 1$, the only column that has to be checked is column
number 1. Its $k$th entry is $b_{k,1}-cb_{k,2}$ be its row. The
condition on $2\times 2$ minors of the original matrix implies
that $|cb_{k,2}-b_{k,1}|\le 1$. The conclusion follows. On the
other hand in the row (from \eqref{E:II}) which started with
$(-d,1)$ we get $(-1,1)$, and in the row which started with
$(c,1)$ we get $(0,1)$. Reordering the coordinates of
$\ell_\infty^m$ (if necessary) we get that the space $X$ has a
basis of the form
\begin{equation}\label{E:Xmodified}
(b_{ij})=\left(
\begin{array}{ccccc}
1 & 0 & 0 &\dots & 0\\
0 & 1 & b_{2,3} &\dots & b_{2,n}\\
0 & 0 & 1 &\dots & 0\\
\vdots & \vdots & \vdots & \ddots & \vdots\\
0 & 0 & 0 &\dots & 1\\
b_{n+1,1} & b_{n+1,2} & * & \dots & *\\
b_{n+2,1} & b_{n+2,2} & * & \dots & *\\
\vdots & \vdots & \vdots & \ddots & \vdots\\
b_{m,1} & b_{m,2} & * & \dots & *
\end{array}\right).
\end{equation}
satisfying the conditions: {\bf (1)}~The absolute values of all
minors do not exceed $1$; {\bf (2)} $|b_{n+1,1}|=|b_{n+1,2}|=1$.
Consider the matrix $D$ obtained from this matrix in the following
way: we keep the values of $b_{n+1,1}$, $b_{n+1,2}$ and the
entries in the first $n$ rows, with the exception of $b_{2,3},
\dots, b_{2,n}$, and let all other entries equal to $0$.
\medskip

The matrix $D$ satisfies the following condition: if some minor of
$D$ is non-zero, then the corresponding minor of
\eqref{E:Xmodified} is its sign. By the results and the discussion
in \cite{Ost03} and \cite{Ost04}, the image of $B_\infty^m$ in $X$
whose kernel is the orthogonal complement of $D$ is a minimal
volume projection which is not a parallelepiped. The extension
property of $\ell_\infty^m$ implies that this image is an MVSE.
\medskip

To prove the result for a general not necessarily polyhedral space
$X$, consider the following polyhedral space $Y$: its unit ball is
the intersection of the parallelepiped corresponding to a lower
Auerbach basis $\{x_i\}$ of $X$ with whose half-spaces, which
correspond to supporting hyperplanes to $B_X$ at midpoints of
sides of the regular hexagon which is the intersection of $B_X$
with the linear span of $x_1,x_2$. As we have just proved the
space $Y$ has a non-parallelepipedal MVSE. Since there is a
minimal-volume parallelepiped containing $B_X$ which contains
$B_Y$, each MVSE for $Y$ is an MVSE for $X$.
\end{proof}

\begin{remark}{Remark} Theorem \ref{T:main} solves Problem 6 posed
in \cite[p.~118]{Ost08b}.
\end{remark}

\section{Comparison of the class of spaces having non-parallelepi\-pe\-dal MVSE with different classes of Banach spaces
}

\subsection{MVSE for $\ell_1^n$}

Our first purpose is to apply Theorem \ref{T:main} to analyze MVSE
of classical polyhedral spaces. For $\ell_\infty^n$ the situation
is quite simple: their unit balls are parallelepipeds and are the
only MVSE for $\ell_\infty^n$. It turns out that the space
$\ell_1^3$ has non-parallelepipedal MVSE, and that for many other
dimensions parallelepipeds are the only MVSE for $\ell_1^n$. To
find more on the problem: characterize $n$ for which the space
$\ell_1^n$ has non-parallelepipedal MVSE, one has to analyze known
results on the Hadamard maximal determinant problem, see
\cite{OS07} for some of such results and related references. In
this paper we make only two simple observations:

\begin{proposition}\label{P:Hadam} If $n$ is such that there exists a
Hadamard matrix of size $n\times n$, then each MVSE for $\ell_1^n$
is a parallelepiped
\end{proposition}

\begin{proof}{Proof} Each upper Auerbach basis for $\ell_\infty^n$
in
such dimensions consists of columns of Hadamard matrices. Hence
their biorthogonal functionals are also (properly normalized)
Hadamard matrices. It is easy to see that any two of them span in
$\ell_1^n$ a subspace isometric to $\ell_1^2$.
\end{proof}

\begin{proposition}\label{P:l_1^3} The $3$-dimensional space $\ell_1^3$ has a non-parallelepipedal MVSE.
\end{proposition}

\begin{proof}{Proof} The columns of the matrix
$$
\left(
\begin{array}{crr}
1 & 1 & 1\\
1 & 1 & -1\\
1 & -1 & 1
\end{array}\right)
$$
form an upper Auerbach basis in $\ell_\infty^3$. The columns of
the matrix
$$
\left(
\begin{array}{rrr}
0 & \frac12 & \frac12\\
\frac12 & 0 & -\frac12\\
\frac12 & -\frac12 & 0
\end{array}\right)
$$
form a biorthogonal system of this upper Auerbach basis. It is
easy to check that the first two vectors of the biorthogonal
system span a regular hexagonal subspace in $\ell_1^3$.
\end{proof}

\subsection{The shape of MVSE and presence of a $1$-complemented
regular hexagonal space}

It would be useful to characterize spaces having
non-parallelepipedal MVSE in terms of their complemented
subspaces. The purpose of this section is to show that one of the
most natural approaches to such a characterization fails. More
precisely, we show that the presence of a $1$-complemented
subspace whose unit ball is linearly equivalent to a regular
hexagon neither implies nor follows from the existence of a
non-parallelepipedal MVSE.

\begin{proposition} There exist spaces having $1$-complemented subspaces whose unit balls are regular
hexagons but such that each of their MVSE is a parallelepiped.
\end{proposition}

\begin{proof}{Proof} Let $X$ be the $\ell_1$-sum of a regular hexagonal
space and a one-dimensional space.
\medskip

(1) The unit ball of the space does not have other sections
linearly equivalent to regular hexagons. This statement can be
proved using the argument presented immediately after equation
\eqref{E:vec2}.
\medskip

(2) Assume that the that the vertices of $B_X$ have coordinates
$\pm(0,0,1)$, $\pm(1,0,0)$,
$\pm\left(\frac12,\pm\frac{\sqrt{3}}2,0\right)$. Denote by $H$ the
hyperplane containing $(1,0,0)$ and $(0,1,0)$. We show that a
lower Auerbach basis cannot contain two vectors in $H$.

In fact, an easy argument shows that the volume of a
parallelepiped of the form $M(\{x_i\}_{i=1}^3)$ containing $B_X$
and such that $x_1,x_2\in H$ is at least $4\sqrt{3}$. On the other
hand, it is easy to check that the volume of a minimal-volume
parallelepiped containing $B_X$ is $\le 2{\sqrt{3}}$.
\end{proof}

\begin{remark}{Remark} The argument of
\cite[pp.~393--395]{Ost04} implies that $\ell_\infty$-sums of a
regular hexagonal space and any space have non-parallelepipedal
MVSE.
\end{remark}

\begin{proposition}\label{P:no_1-compl} The existence of a lower Auerbach basis with two elements of
it spanning a regular hexagonal subspace does not  imply the
presence of a $1$-complemented regular hexagonal subspace.
\end{proposition}

\noindent{\sc Proof.} Consider the subspace $X$ of $\ell_\infty^4$
described by the equation $x_1+x_2+x_3+x_4=0$. The fact that this
space has a non-parallelepipedal MVSE follows immediately from the
fact that the columns of the matrix
$$
\left(
\begin{array}{rrr}
1 & 0 & 0\\
0 & 1 & 0\\
0 & 0 & 1\\
-1 & -1 & -1
\end{array}\right)
$$
form a lower Auerbach basis in $X$ (see the argument after the
equation \eqref{Z}) and any two of them span a subspace whose unit
ball is linearly equivalent to a regular hexagon.
\medskip

So it remains to show that the space $X$ does not have
$1$-complemented subspaces linearly equivalent to a regular
hexagonal space. It suffices to prove the following lemmas. By a
{\it support} of a vector in $\ell_\infty^m$ we mean the set of
labels of its non-zero coordinates.

\begin{lemma}\label{L:vertices} The only two-dimensional subspaces of $X$ which have
balls linearly equivalent to regular hexagons are the spaces
spanned by vectors belonging to $X$ and having intersecting
two-element supports.
\end{lemma}

\begin{lemma}\label{L:no_1-compl} Two-dimensional subspaces satisfying the conditions
of Lemma \ref{L:vertices} are not $1$-complemented in $X$.
\end{lemma}

\begin{proof}{Proof of Lemma \ref{L:vertices}} Consider a
two-dimensional subspace $H$ of $X$. It is easy to check that if
the unit ball of $H$ is a hexagon, then each extreme point of the
hexagon  is of the form: two coordinates are $1$ and $-1$, the
remaining two are $\alpha$ and $-\alpha$ for some $\alpha$
satisfying $|\alpha|\le 1$. Two different forms cannot give the
same extreme point unless the corresponding value of $\alpha$ is
$\pm 1$. Also two points of the same type cannot be present unless
the corresponding values of $\alpha$ are $+1$ and $-1$. Since
$B_H$ is a hexagon, there are $3$ pairs of extreme points. First
we consider the case when none of $\alpha_i$, $i=1,2,3$,
corresponding to an extreme point is $\pm 1$. Then $\pm 1$ either
form a cycle or a chain in the sense shown in
\eqref{E:cyclechain}.
\begin{equation}\label{E:cyclechain}
\left(\begin{array}{rrr} 1 &\alpha_2 & -1\\
-1 & 1 & \alpha_3\\
\alpha_1 & -1 & 1\\
-\alpha_1 & -\alpha_2 & -\alpha_3
\end{array}\right)\hbox{ or }
\left(\begin{array}{rrr} 1 &\alpha_2 &\alpha_3\\
-1 & 1 & -\alpha_3\\
\alpha_1 & -1 & 1\\
-\alpha_1 & -\alpha_2 & -1
\end{array}\right)
\end{equation}

If they form a cycle, by considering determinants (as after
\eqref{E:vec2}) with other unit vectors we get: all involved
$\alpha_i$ are zeros. Thus we get a subspace of the form described
in the statement of the lemma.
\medskip

We show that $\pm1$ cannot form a chain as in the second matrix in
\eqref{E:cyclechain} by showing that in such a case they cannot be
linearly dependent. In fact, multiplying the first column by
$\alpha_3$ and subtracting the resulting column from the third
column we get
$$
\left(\begin{array}{rrr} 1 &\alpha_2 & 0\\
-1 & 1 & 0\\
\alpha_1 & -1 & 1-\alpha_1\alpha_3\\
-\alpha_1 & -\alpha_2 & -1+\alpha_1\alpha_3
\end{array}\right).
$$
It is clear that this matrix has rank $3$.
\medskip

It remains to consider the case when some of the extreme points
have all coordinates $\pm1$. Assume WLOG that one of the extreme
points is $(1,1,-1,-1)$. If there is one more $\pm 1$ extreme
point (different from $(-1,-1,1,1)$), the section is a
parallelogram.\medskip

If the other extreme point is not a $\pm 1$ point, then it has
both $+1$ and $-1$ either in the first two positions or in the
last two positions (otherwise it is not an extreme point). In this
case the section is also a parallelogram, because the norm on
their linear combinations is just the $\ell_1$-norm
\end{proof}

\begin{proof}{Proof of Lemma \ref{L:no_1-compl}} In fact, assume
without loss of generality that we consider a two dimensional
subspace spanned by the vectors
$$\left(
\begin{array}{r} 1\\-1\\0\\0
\end{array}\right) \hbox{ and }
\left(\begin{array}{r} 0\\1\\-1\\0\end{array}\right).
$$
We need to show that there is no vector in this subspace such that
projecting the vector
$$\left(
\begin{array}{r} 0\\0\\1\\-1
\end{array}\right)$$
onto it we get a projection of norm $1$ on $X$. Assume the
contrary. Let
$$\left(
\begin{array}{c} a\\b-a\\-b\\0
\end{array}\right)$$
be the desired vector. The condition that the images of the
vectors
$$\left(
\begin{array}{r} 1\\-1\\\pm1\\\mp 1
\end{array}\right)
$$
under the projection are vectors of norm $\le 1$ implies
immediately that $a=(b-a)=0$. hence $a=b=0$. Now we get a
contradiction by projecting the vector
\[\left(
\begin{array}{r} 1\\1\\-1\\-1
\end{array}\right);\]
its image has norm $2$.~\rule{0.5em}{0.5em}
\end{proof}

\begin{small}

{~}
\vfill

\noindent{\it Department of Mathematics and Computer Science\\
St. John's University\\
8000 Utopia Parkway\\
Queens, NY 11439\\
USA\\
e-mail: {\tt ostrovsm@stjohns.edu}}

\end{small}
\end{large}

\begin{thebibliography}{12}


\bibitem{Aue30} H.~Auerbach, {\it On the area of convex curves with conjugate
diameters} (in Polish), Ph.~D. thesis, University of Lw\'ow, 1930.

\bibitem{Ban32} S.~Banach, {\it Th\'eorie des op\'erations
lin\'eaires}, Monografje Matematyczne I, Warszawa, 1932.

\bibitem{Day47} M.~M.~Day, Polygons circumscribed about closed convex
curves, {\it Trans. Amer. Math. Soc.}, {\bf 62} (1947), 315--319.

\bibitem{Erd99} R.~M.~Erdahl, Zonotopes, dicings, and
Voronoi's conjecture on parallelohedra, {\it European J. Combin.},
{\bf 20} (1999), 427--449.

\bibitem{GMP96} Y.~Gordon, M.~Meyer, A.~Pajor,
Ratios of volumes and factorization through $\ell_\infty$, {\it
Illinois J. Math.}, \textbf{40} (1996), 91--107.

\bibitem{JL01} W.B.~Johnson and J.~Lindenstrauss, Basic
concepts in the geometry of Banach spaces, in: {\it Handbook of
the geometry of Banach spaces} (W.B.~Johnson and J.~Lindenstrauss,
Eds.) Vol. {\bf 1},  Elsevier, Amsterdam, 2001, pp.~1--84.

\bibitem{McM75} P.~McMullen, Space tiling zonotopes, {\it Mathematika}, {\bf 22}  (1975), no.
2, 202--211.

\bibitem{OS07} W.~P.~Orrick, B.~Solomon, Large-determinant
sign matrices of order $4k+1$, {\it Discrete Math.}, {\bf 307}
(2007), no. 2, 226--236; {\tt ArXiv: math.CO/0311292}.

\bibitem{Ost96} M. I. Ostrovskii, Generalization of projection
constants: sufficient enlargements, {\it Extracta Math.}, {\bf 11}
(1996), no. 3, 466-474.

\bibitem{Ost98} M.~I.~Ostrovskii, Projections in normed linear
spaces and sufficient enlargements, {\it Archiv der Mathematik},
{\bf 71} (1998), no. 4, 315--324; {\tt ArXiv:~math.FA/0203085}.

\bibitem{Ost03}
M.~I.~Ostrovskii, Minimal-volume projections of cubes and totally
unimodular matrices, {\it Linear Algebra and Its Applications},
{\bf 364} (2003), 91--103.

\bibitem{Ost04}
M.~I.~Ostrovskii, Sufficient enlargements of minimal volume for
two-dimensional normed spaces, {\it Math. Proc. Cambridge Phil.
Soc.}, {\bf 137} (2004), 377-396.

\bibitem{Ost08}
M.~I.~Ostrovskii, Sufficient enlargements of minimal volume for
finite dimensional normed linear spaces, {\it J. Funct. Anal.},
{\bf 255} (2008) no. 3, 589--619; {\tt arXiv:0811.1701}.

\bibitem{Ost08b} M.~I.~Ostrovskii,
Sufficient enlargements in the study of projections in normed
linear spaces, {\it Indian Journal of Mathematics}, Golden Jubilee
Year Volume, 2008 (Supplement), Proceedings, Dr. George Bachman
Memorial Conference, The Allahabad Mathematical Society,
pp.~105-122.

\bibitem{PS91} A.~Pe\l czy\'nski and S.~J.~Szarek, On parallelepipeds
of minimal volume containing a convex symmetric body in
${\mathbb{R}}^n$, {\it Math. Proc. Cambridge Phil. Soc.} {\bf 109}
(1991), 125--148.

\bibitem{Pli95} A.~M.~Plichko, On the volume
method in the study of Auerbach bases of finite-dimensional normed
spaces, {\it Colloq. Math.}, {\bf 69} (1995), 267--270.

\bibitem{Sch93} R.~Schneider,
{\it Convex Bodies: the Brunn--Minkowski Theory}, Encyclopedia of
Mathematics and its Applications, vol. {\bf 44}, Cambridge
University Press, 1993.

\bibitem{Tay47} A.~E.~Taylor, A geometric theorem and its application to
biorthogonal systems,  {\it Bull. Amer. Math. Soc.}, {\bf 53}
(1947), 614--616.

\end{thebibliography}
\end{document}